\documentclass[12pt]{amsart}
\usepackage{color,graphicx,enumerate,wrapfig,amssymb}
\usepackage{pxfonts}
\usepackage{subfigure}
\usepackage{eucal}
\usepackage{amsmath}

\newcommand{\R}{{\mathbb R}}
\newcommand{\Rn}{{\mathbb R}^n}

\newcommand{\re}{\mathbb{R}}

\newcommand{\La}{\Delta}

\newcommand\norm[1]{\Arrowvert {#1}\Arrowvert}

\newcommand{\RN}[1]{%
  \textup{\uppercase\expandafter{\romannumeral#1}}
}

\newtheorem{theorem}{Theorem}[section]

\newtheorem{cor}[theorem]{Corollary}

\newtheorem{definition}{Definition}[section]

\newtheorem{lem}[theorem]{Lemma}

\newtheorem{prop}[theorem]{Proposition}
\newtheorem{remark}[theorem]{Remark}

\title [Double Obstacle Problems]{The Regularity Theory for the Double Obstacle Problem}

\author{Ki-ahm Lee}
\address{Department of Mathematical Sciences,
Seoul National University, Seoul 08826, Korea.
Center for Mathematical Challenges,
Korea Institute for Advanced Study, Seoul 02455, Korea}
\email{kiahm@snu.ac.kr}
\author{Jinwan Park}
\address{Department of Mathematical Sciences,
Seoul National University, Seoul 08826, Korea.}
\email{jinwann@snu.ac.kr}
\author[Henrik Shahgholian ]{Henrik Shahgholian}
\address{Department of Mathematics, KTH Royal Institute of Technology,
  100~44  Stockholm, Sweden}
\email{henriksh@kth.se}

\thanks{\textup{K. Lee has been supported by the National Research Foundation of Korea (NRF) grant funded by the Korea government (MSIP) (No.2015R1A4A1041675).
K. Lee also holds a joint appointment with the Research Institute of Mathematics of Seoul National University.
J. Park has been supported by National Research Foundation of Korea (NRF) grant funded by the Korean government (Global Ph.D. Fellowship). 
H. Shahgholian has been supported in part by Swedish Research Council. 
}}

\begin{document}
\maketitle 
\begin{abstract}
In this paper, we prove local $C^{1}$ regularity of free boundaries 
for the double obstacle problem with an upper obstacle $\psi$,
\begin{align*}
\Delta u &=f\chi_{\Omega(u) \cap\{ u< \psi\} }+ \Delta \psi \chi_{\Omega(u)\cap \{u=\psi\}}, \qquad u\le \psi \quad \text { in } B_1,
\end{align*} 
where $\Omega(u)=B_1 \setminus \left( \{u=0\} \cap \{ \nabla u =0\}\right)$ under a thickness assumption for $u$ and $\psi$.

\end{abstract}


\section{Introduction and Main Results}

\subsection{Background}

 In the last five decades, the classical (one-sided) obstacle problem has been subject of intense studies.  On the other hand, the corresponding two-sided counterpart of this problem (the double-obstacle problem)  has not attracted the same interest, and hence there are much less known results concerning this problem. A particular problem, of interest to us, is the regularity of the free boundary 
 for this problem, which has not been addressed in the literature. 
 Readers may consult  \cite{MR} for a review  on the problem and also a list over existing literature. For a recent regularity theory for  a particular case of  this problem we refer to work of G. Aleksanyan \cite{Ale}, where she considers the global homogeneous solutions to the double obstacle problem, with homogeneous obstacles. Another interesting paper on the topic is \cite{DMV}.
 
Here we shall consider a double obscale problem which relaxes one of the obstacles, see \eqref{main eq} here below. Our result is very close to the well-known regularity theory of L. Caffarelli for the obstacle problem \cite{Caf}, and also the no-sign obstacle problem due to Caffarelli-Karp-Shahgholian \cite{CKS}. 

To set the scene for our study, we consider the double obstacle problem 
 with a function $$\psi \in C^{1,1}(B_1) \cap  C^{2,1}(\overline{\Omega(\psi)}),\quad \Omega(\psi)=B_1 \setminus \left( \{\psi=0\} \cap \{ \nabla \psi =0\}\right)$$ in a domain $B_1\subset \R^n$  ($n\geq 2$):\footnote{For this formulation we refer to \cite{FS15}. Also the solution is allowed to penetrate through the lower obstacle. This is usually referred to as no-sign obstacle problem.}
\begin{align}
\Delta u =f\chi_{\Omega(u) \cap\{ u< \psi\} } + \Delta \psi \chi_{\Omega(u)\cap \{u=\psi\}}, \qquad &u\le \psi \quad \text { in } B_1,\label{main eq}
\end{align}
with
\begin{align*}
\Omega(u)&=B_1 \setminus \left( \{u=0\} \cap \{ \nabla u =0\}\right),
\end{align*}
where $f\in C^{0,1}(B_1).$ The function $\psi$ is called \emph{ the upper obstacle}.

\subsection{Notation}
 We will use the following notations throughout the paper. 
$$\begin{array}{ll}
C, C_0, C_1       &\hbox{generic constants }\cr 
\chi_E            &\hbox{the characteristic function of the set } E, (E \subset \Rn)\cr 
\overline E       &\hbox{the closure of } E\cr 
\partial E        &\hbox{the  boundary of a set }  E \cr
|E|               &n-\hbox{dimensional Lebesgue measure of the set } E\cr 
B_r(x), B_r   \qquad    &\{y\in \Rn: |y - x|<r\}, \quad B_r(0) \cr
\Omega(u), \Omega(\psi) & \hbox{see Equation } \eqref{main eq}  \\
\Lambda(u), \Lambda(\psi) & B_1\setminus \Omega(u), B_1\setminus \Omega(\psi) \\
\Omega^\psi(u)  & B_1\setminus \left( \{u=\psi\} \cap \{ \nabla u =\nabla \psi\} \right)= B_1\setminus \{u=\psi\}= \{u>\psi\}  \\
                         & (u \le \psi \hbox{ implies } \{u=\psi\} \cap \{ \nabla u =\nabla \psi\}=\{ u =\psi\}.) \\
\Lambda^\psi(u) & B_1\setminus \Omega^\psi(u)=\{u=\psi\} \\
\Gamma(u),\Gamma^\psi(u)    & \partial \Lambda(u)\cap B_1,\partial \Lambda^\psi(u) \cap B_1   \hbox{ }\\
\Gamma^d(u)         & \Gamma(u)\cap \Gamma^\psi(u)\\ 
u^+,u^-           &\max (u,0),   \max (-u,0)\cr
\| u \|_{\infty,E} & \hbox{the supremum norm of the function $u$ on the set $E$}\cr
\partial_{\mathbf{\nu}}, \partial_{\nu e}& \hbox{first and second  directional  derivatives }\cr
P_r(M), P_\infty (M)  & \hbox{see Definition } \ref{loc sol},  \ref{glo sol}\cr
\delta_r(u,x), \delta_r(u) &\hbox{see  Definition } \ref{thi}
\end{array}
$$ 

\subsection{Preliminaries}

Let $u $ be a solution of \eqref{main eq} in $B_r$. Then a \emph{rescaling function} of $u$ at $x_0$ with $\lambda>0$ is 
$$u_\lambda(x)=u_{\lambda,x_0}(x):=\dfrac{u(x_0+\lambda x)-u(x_0)}{\lambda^2}, \quad x\in B_{r/\lambda}.$$
The $C^{1,1}$-regularity of solution $u$ (Theorem \ref{rem opt}) implies the uniform boundedness of $C^{1,1}$-norm of the rescaling functions and the uniform boundedness gives limit functions which are called a blowup and a shrink-down. More precisely, if $u$ is a solution of \eqref{main eq} in $B_r$, then for a sequence $\lambda_i \to 0$, there exists a subsequence $\lambda_{i_j}$  of $\lambda_i$ and $u_0 \in C^{1,1}_{loc}(\re^n)$ such that 
$$u_{\lambda_{i_j}} \rightarrow u_0 \text{ in } C^{1,\alpha}_{loc}(\R^n) \quad \text{ for any } 0<\alpha<1.$$ Such $u_0$ is called a \emph{blowup of $u$ at $x_0$}.
 Let $u$ be a solution of \eqref{main eq} in $\re^n$. Then, for a sequence $\lambda_i \to \infty$, there exists a subsequence $\lambda_{i_j}$  of $\lambda_i$  and $u_0 \in C^{1,1}_{loc}(\re^n)$
such that
$$u_{\lambda_{i_j}} \rightarrow u_\infty \text{ in } C^{1,\alpha}_{loc}(\R^n) \quad \text{ for any } 0<\alpha<1.$$ Such $u_\infty$ is called a \emph{shrink-down of $u$ at $x_0$}.

\begin{definition}\label{thi}
We denote by $\delta_r(u,x)$ the thickness of $\Lambda(u)$ on $B_r(x)$, i.e.,
$$\delta_r(u,x):=\dfrac{\text{MD}(\Lambda(u) \cap B_r(x))}{r},$$ 
where MD($A$) is the least distance between two parallel hyperplanes containing $A$. We will use the abbreviated notation $\delta_r(u)$ for $\delta_r(u,0)$.
\end{definition}

\begin{remark}
The thickness $\delta_r$ satisfies $\delta_1(u_r)=\delta_r(u)$, where $u_r=u_{r,0}$. Thus, by the fact that $\limsup_{r\to 0} \Lambda(u_r) \subset \Lambda (u_0)$, we have
$$\limsup_{r\to 0} \delta_r(u) \le \delta_1(u_0).$$
Hence the thickness assumption \eqref{thic assum} in Theorem \ref{reg lo1} implies 
$$\min{\left\{\delta_r(u_0), \delta_r(\psi_0)\right\}}\ge \epsilon_0 \quad \forall r>0,$$
for any blowups $u_0$ and $\psi_0$ of $u$ and $\psi$ at $0$, respectively.
\end{remark}

In order to state our main results, we define classes of local and global solutions of the problem.

\begin{definition} \label{loc sol}(Local solutions)
 We say  a function $u$ belongs to the class
$P_r(M)$ $(0<r<\infty)$, if 
$u$ satisfies :
\begin{enumerate}[(i)]
\item   $\Delta u =f\chi_{\Omega(u) \cap\{u <\psi\} } + \La \psi \chi_{\Omega(u) \cap \{u=\psi\}}, \quad u\le \psi \quad \text { in } B_r,$ 
\item  $\|D^2 u \|_{\infty,B_r} \leq M$,
\item $0\in \Gamma^d(u),$
\end{enumerate}
where $f\in C^{0,1}(B_r)$ and $\psi \in C^{1,1}(B_r)\cap C^{2,1}(\overline{\Omega(\psi)})$.
\end{definition}

\definition \label{glo sol}(Global solutions)
 We say  a function $u$ belongs to the class
$P_\infty(M)$, if 
$u$ satisfies with a constant $a>1$:
\begin{enumerate}[(i)]
\item  $\Delta u = \chi_{\Omega(u) \cap \{u<\psi\}} + a \chi_{\Omega(u) \cap \{u=\psi\}}$, \quad $ u \le \psi$ \quad  in $ \re^n$,
\item  $\Delta \psi =a \chi_{\Omega(\psi)}$ \text{ in } $\re^n$,
\item  $\| D^2 u \|_{\infty,\re^n} \leq M$,
\item  $0\in \Gamma(u).$
\end{enumerate}
\enddefinition

\subsection{Main Results}

\begin{theorem} \label{reg lo1}(Regularity of free boundaries)
 Let $u \in P_1(M)$ with an upper obstacle $\psi$ such that 
$$0\in \partial \Omega(\psi), \quad \lim_{x \to 0, x\in \Omega(\psi)} \Delta \psi(x)>f(0), \quad f\ge c >0 \text{ in } B_1,$$
and 
$$ \inf{\left\{\Delta \psi,  \Delta \psi-f\right\}} \ge c>0 \text { in }\Omega(\psi).$$
Suppose
\begin{equation}\label{thic assum}
\min{\left\{\delta_r(u), \delta_r(\psi)\right\}}\ge \epsilon_0 \quad \forall r< 1/4.
\end{equation}
 Then there is $r_0=r_0(u,\psi)>0$ such that $\Gamma(u) \cap B_{r_0}$ and $\Gamma^\psi(u) \cap B_{r_0}$ are $C^1$ graphs. 
\end{theorem}

\section{Standard Results}

\subsection{Optimal regularity}
\textup{The double obstacle problem \eqref{main eq} fall under a more general class of problems, studied in  \cite{FS14, IM}, where   optimal regularity of solutions for the larger class is already proven. Hence we shall only state the result without repeating the proof. 
}

\begin{theorem}\label{rem opt}(Optimal regularity)
Let $u$ be a $W^{2,n}$ solution of \eqref{main eq} in $B_1$, with $f\in C^{0,\alpha}(B_1)$ and $\psi \in C^{1,1}(B_1)$. Then 
$$\norm{D^2 u}_{\infty,B_{1/2}}\le C$$
where $C>0$ is a universal constant.
\end{theorem}

\begin{proof}
Since $\psi \in C^{1,1}(B_1)$, we obtain that $|D^2 u|$ is bounded a.e. on $\{u=\psi\}$. Then the solution $u$ of \eqref{main eq} satisfies 
\begin{displaymath}
\left\{ \begin{array}{ll}
\Delta u =f & \text{ a.e. in } B_1\cap \left( \Omega(u)\cap \{u< \psi\} \right) ,\\
|D^2 u|\le K & \text{ a.e. in }B_1 \setminus \left( \Omega(u)\cap \{u< \psi\} \right),
\end{array} \right.
\end{displaymath}
for a positive constant $K$;
i.e., $u$ is in the general classes defined in \cite{FS14, IM}. By the $C^{1,1}$ regularity theory in the papers (more specifically, Theorem 1.2 of \cite{FS14}, Theorem 2.1 of \cite{IM}), we obtain the $C^{1,1}$ regularity of the solution $u$.
\end{proof}

\subsection{Non-degeneracy}
\textup{Non-degeneracy is one of the important properties of the obstacle problem. In particular, it implies that the blowups of the solutions are still solutions to the problem, and that they do not flatten out to the identically zero function. A second consequence of the non-degeneracy  along with the optimal growth, is that the  Lebesgue measure of the free boundary  is zero.
}

\begin{lem}\label{nond}
Let $u\in P_1(M)$. If $f\ge c >0$ in $B_1$ and $\Delta \psi \ge c >0$ in $\Omega(\psi)$, then
\begin{equation*}
\sup_{\partial B_r(x)}u \ge u(x)+\frac{c}{8n}r^2, \quad x\in \overline{\Omega(u)}\cap B_1,
\end{equation*}
for any $B_r(x) \Subset B_1$.
\end{lem}

\begin{proof}
(i) Let $x^0\in \Omega(u) \cap B_1$ be such that $u(x^0)>0$. Consider the auxiliary function
$$\phi(x):=u(x)-u(x^0)-\dfrac{c}{2n}|x-x^0|^2.$$
Due to $\Omega(u)\cap \{u=\psi\} \subset \Omega(\psi)$ and the assumptions for $f$ and $\Delta \psi$, we obtain 
\begin{equation}\label{eq non}
\Delta u=f\chi_{\Omega(u) \cap \{u<\psi\}}+\Delta \psi \chi_{\Omega(u)\cap \{u=\psi\}}\ge c \quad \text{ in }\Omega(u).
\end{equation}
Hence we have
$$\Delta \phi \ge \Delta u -c \ge 0 \text{ on } B_r(x^0)\cap \Omega(u).$$
Thus, by the maximum principle, $\phi$ attains its maximum on $\partial (B_r(x^0)\cap \Omega(u))$. Hence
$$0=\phi(x^0)\le \sup_{\partial(B_r(x^0)\cap \Omega(u))}\phi.$$
Moreover, $\phi(x)=-u(x^0)-\frac{c}{2n}|x-x^0|^2<0$ on $\partial \Omega(u)$, which  implies that  
$$0\le \sup_{\partial B_r(x^0)\cap \Omega(u)}\phi,$$
and 
\begin{equation*}
\sup_{\partial B_r(x^0)}u \ge u(x^0)+\frac{c}{2 n}r^2. 
\end{equation*}

(ii) Now, let $x^0\in \Omega(u) \cap B_1$ and assume $u(x^0)\le 0$. Suppose that there is a point $x^1 \in B_{r/2}(x^0)$ such that $u(x^1) >0.$ Then we obtain
$$\sup_{B_r(x^0)} u  \ge \sup_{B_{r/2}(x^1)} u \ge u(x^1)+ \frac{c}{8n}r^2\ge u(x^0)+ \frac{c}{8n}r^2. $$
Since $u$ is subharmonic, 
$$\sup_{\partial B_r(x^0)} u  =\sup_{B_r(x^0)} u \ge u(x^0)+ \frac{c}{8n}r^2. $$

Suppose that $u(x) \le 0$ in $B_{r/2}(x^0).$ By the maximum principle, we know that $u(x) \equiv 0$ in $B_{r/2}(x^0)$ or $u(x)<0$ in $B_{r/2}(x^0)$. The first case is impossible, since $x^0 \in \Omega(u)$. The second case implies that $\La u\ge c$ in $B_{r/2}(x^0)$. By using the auxiliary function $w(x)=u(x)-\frac{c |x-x_0|^2}{2n}$, we obtain 
$$\sup_{\partial B_{r/2}(x^0)} w  \ge\sup_{B_{r/2}(x^0)} w  \ge  w(x^0)=u(x^0),$$
and thus
$$\sup_{\partial B_{r/2}(x^0)} u \ge u(x^0)+ \frac{c}{8 n}r^2. $$
Since $u$ is subharmonic, we have the desired inequality.

Let $x^0\in \partial \Omega(u) \cap B_1$ and take a sequence of points $x^j \in \Omega(u)$ such that $x^j \to x^0$ as $j \to \infty$. By passing to the limit as $j$ goes to $\infty$, we have the desired inequality for $x^0\in \overline{\Omega(u)} \cap B_1.$ 
\end{proof}

\emph{
By using the non-degeneracy for $u$, we have the local porosity for $\partial \Lambda(u)=\Gamma(u)$. Moreover, the porosity implies $\Gamma(u)$ has a Lebesgue measure zero (see Section 3.2.1 of \cite{PSU}).
}
\begin{lem}\label{Leb fb}[Lebesgue measure of $\Gamma(u)$]
Let $u\in P_1(M)$. If $f\ge c >0$ in $B_1$ and $\Delta \psi \ge c >0$ on $\Omega(\psi)$,
then $\Gamma(u)$ has a Lebesgue measure zero.
\end{lem}

\begin{remark}\label{nonde for v}
By the non-degeneracy, we know that $0\in \Gamma(u_0)$ where $u_0$ is a blowup of $u\in P_1(M)$ (see Theorem 3.17 (iv) of \cite{PSU}). However, we do not have any information whether $0\in \Gamma^{\psi_0}(u_0)$, where $\psi_0$ is a blowup of the upper obstacle $\psi$ of $u$ (which is the reason why we assume (iv) in Definition \ref{glo sol} and not $0 \in \Gamma^d(u)=\Gamma(u) \cap \Gamma^\psi(u)$).\\

However, we have $0\in \Gamma^{\psi_0}(u_0)$, under the additional assumption for $u\in P_1(M)$, $0\le u$ in $B_1$ and $\Delta \psi - f\ge c >0$ and $\Delta \psi \ge c >0$ in $\Omega(\psi)$. If we assume $0\le u$ in $B_1$, then $u$ is a solution of
$$\Delta u =f\chi_{\{0< u< \psi\} } + \Delta \psi \chi_{ \{0<u=\psi\}}, \qquad 0\le u\le \psi \quad \text { in } B_1,$$
and $v:=\psi-u$ is a solution of 
$$\Delta v=\left(\Delta \psi-f \right) \chi_{\{0< v <\psi\}}+\Delta \psi \chi_{\{0<v=\psi\}}, \qquad 0\le v\le \psi \quad \text{ in } B_1.$$
Since $\La \psi-f$ lies in $C^{0,1}(\overline{\Omega(\psi)})=C^{0,1}(\overline{\{\psi>0\}})$ but not in $C^{0,1}(B_1)$, we know that $v$ does not belong to $P_1(M)$. However, $0\le v \le \psi$ implies $\{v>0\} \subset \{\psi>0\}$ and
$$\La v=\left(\Delta \psi -f \right)\chi_{\{0<v<\psi\}}+\La \psi \chi_{\{0<v=\psi\}}\ge c \quad \text{ in } \Omega(v)=\{v>0\},$$
provided 
$\La \psi-f \ge c$ and $\La \psi \ge c$ in $\Omega(\psi)=\{\psi>0\}.$ Then the rest of the proof for the non-degeneracy for $v$ is a repetition of the arguments in the proof of Lemma \ref{nond}. 
Thus, we have the non-degeneracy for $v$ and moreover $0\in \Gamma(v_0)=\Gamma^{\psi_0}(u_0)$ and $|\Gamma(v)|=|\Gamma^\psi(u)|=0.$
\end{remark}

\section{Properties of Global Solutions}\label{sec pro glo}

\textup{In this section, we consider some properties of global solutions with the upper obstacle $\psi=\frac{a}{2}(x_1^+)^2.$}

\subsection{Dimensionality Reduction and Positivity of Global Solutions with the Upper Obstacle $\psi=\frac{a}{2}(x_1^+)^2$}
\textup{In order to discuss dimensionality reduction of global solutions, we introduce Alt-Caffarelli-Friedman (ACF) monotonicity formula which is  an important tool in analysis of regularity of free boundary; see  \cite{ACF}, and also \cite{CS} for a more detailed proof.
}

\begin{theorem}[Alt-Caffarelli-Friedman (ACF) monotonicity formula]\label{ACF} \text{ }
 Let $u_{\pm}$ be continuous functions on $B_1$ such that 
$$u_{\pm} \ge 0, \quad \Delta u_\pm \ge 0, \quad u_+ \cdot u_- = 0 \quad \text{ in } B_1$$
Then the functional
$$r \to \Phi(r)= \Phi(r, u_+,u_-)=\frac{1}{r^4} \int _{B_r} \frac{| \nabla u_+|^2}{|x|^{n-2}}dx\int _{B_r} \frac{| \nabla u_-|^2}{|x|^{n-2}}dx$$
is nondecreasing for $0<r<1$.
\end{theorem}

\begin{theorem}[Equality in ACF monotonicity formula]\label{cas ACF}
 Let $u_{\pm}$ be as in Theorem \ref{ACF} and assume that $\Phi(r_1)=\Phi(r_2)$ for some $0<r_1<r_2<1$. Then either one of the following holds:
\begin{itemize}
\item[(i)] $u_+=0$ in $B_{r_2}$ or $u_-=0$ in $B_{r_2}$;
\item[(ii)] there exists a unit vector $e$ and constants $k_{\pm}>0$ such that
$$u_+(x)=k_+(x\cdot e)^+, \quad u_-(x)=k_-(x\cdot e)^- \quad \text{ in } B_{r_2}.$$
\end{itemize}
\end{theorem}

\begin{lem}\label{pm}
Let $u \in P_1(M)$ with the upper obstacle $\psi=\frac{a}{2}(x_1^+)^2$.
 Then for any unit vector $e$ such that $e\perp e_1$,
$$\Delta (\partial_e u)^{\pm} \ge 0 \quad \text{ in } B_1.$$
\end{lem}

\begin{proof}
Let $e$ be a unit vector such that $e\perp e_1$ and $E:=\{\partial_e u >0 \}$.
Since $\partial_e \psi \equiv 0$, we know that $E \subset \Omega(u) \cap \{u< \psi\}$ ($u\le \psi$ implies $\{u=\psi\}=\left\{\{u=\psi\} \cap \{ \nabla u=\nabla \psi \} \right\} $) and $\La u =1 \text{ on } E$. Consequently, we have $\Delta (\partial_e u)=0 \text{ on } E$ and 
$$\Delta (\partial_e u)^{+} \ge 0 \quad \text{ in } B_1$$
This is left to the reader as an exercise.

We have the same inequality for $(\partial_e u)^{-}$, by using the direction $-e$ instead of $e$.
\end{proof}

\begin{lem}\label{2dim}
Let  $u \in P_\infty(M)$ with the upper obstacle  
$$\psi(x)=\frac{a}{2}(x_1^+)^2 \quad \text{ in } \re^n.$$
Assume that there exists $\epsilon_0>0$ such that
$$ \delta_r(u) \ge \epsilon_0 \quad \forall r>0.$$
Then we have $|Int\Lambda(u)|\neq 0$ and $u$ is two-dimensional, i.e.
$$u(x)=w(x_1,x_2) \quad \forall x\in \re^n,$$
with $\partial_2 w\ge 0$, in an appropriate system of coordinates.
\end{lem}

\begin{proof}

Suppose $|Int \Lambda(u)|=0$. Then, by Lemma \ref{Leb fb}, we have $|\partial \Lambda(u)|=0$ and $|\Lambda (u)|=0$. Thus $u$ is a solution of
$$\Delta u= \chi_{\{u< \psi\}}+a\chi_{\{u=\psi\}} \quad \text{ a.e. in } \re^n.$$
Define $\tilde \psi :=\frac{a}{2}(x_1)^2$. Then $\tilde v:=\tilde \psi-u$ is a solution of
$$\Delta \tilde v= (a-1)\chi_{\{u<\psi\}} \quad \text{ a.e. in } \re^n.$$
Since $\Omega(u) \cap \{u=\psi\} \subset \{x_1\ge 0\}$, we know that $\Delta u \le 1$ a.e. in $\{x_1<0 \}$. On the other hand, $\Delta u =a $ a.e. in $\{x_1<0\} \cap \{u=\tilde \psi\}\cap \{ \nabla u=\nabla \tilde \psi\}$.
 Therefore, we know that $|\{x_1<0\} \cap \{u=\tilde \psi\}\cap \{\nabla u=\nabla \tilde \psi\}|= 0 $. By the definition of $\tilde \psi$ and $\psi$, we obtain $\{u=\tilde \psi\}\cap \{\nabla u=\nabla \tilde \psi\}=\{u= \psi\}\cap \{\nabla u=\nabla \psi\}= \{u=\psi\}$ a.e. in $\re^n$ ($u \le \psi$ implies the last equality). Therefore $\tilde v$ is a solution of
$$\Delta \tilde v= (a-1)\chi_{\Omega(\tilde v)} \quad \text{ a.e. in } \re^n,$$
where $\Omega(\tilde v):= \re^n \setminus \left( \{\tilde v=0\} \cap \{\nabla \tilde v=0\}\right)= \re^n \setminus \left(\{u=\tilde \psi\}\cap \{\nabla u=\nabla \tilde \psi\} \right).$ 

By the definition of $\tilde v$, we know that $0\in \Lambda(\tilde v).$ Suppose $0\in int \Lambda(\tilde v)$. Then there is a ball $B_r$ such that $\tilde v \equiv 0$ and $u \equiv \tilde \psi$ in $\re^n$. Thus we have a contradiction to $\delta_r(u)> \epsilon_0$ for all $r>0$.\\
Suppose $0\in \Gamma(\tilde v)$. Let $u_0, \tilde v_0$ be blowup functions of $u$ and $\tilde v$, respectively, such that $u_0=\tilde \psi-\tilde v_0.$ Then
 $\tilde v_0$ is a solution of
$$\Delta \tilde v_0= (a-1)\chi_{\Omega(\tilde v_0)} \quad \text{  a.e. in } \re^n,$$ and by Theorem 3.22 of \cite{PSU},  we know that
 $\tilde v_0$ is a polynomial or a half-space solution. 
In the both cases, we have a contradiction to $\delta_r(u_0)> \epsilon_0$ for all $r>0$.
Thus, we obtain $$|Int \Lambda(u)|\neq 0.$$ 

Let $u_\infty$ be a shrink-down of $u$ at $0$, then $u_\infty\in P_\infty(M)$ with the upper obstacle $\psi=\frac{a}{2}(x_1^+)^2$ and the thickness assumption,
$$\min\left\{ \delta_r(u_\infty), \delta_r(\psi)\right\}> \epsilon_0 \quad \forall r>0.$$
Hence we also have
$$|Int \Lambda(u_\infty)| \neq 0.$$

For $r>0$ and a unit vector $e$, we define
$$\phi_e(r,u):=\Phi(r,(\partial_e u)^+,(\partial_e u)^-).$$
By $W^{2,p}$ convergence $u_{r_j} \to u_\infty$, we have
$$\phi_e(r,u_\infty)=\lim_{j\to \infty} \phi_e(r,u_{r_j}).$$
Additionally, we obtain the rescaling property,
$$\phi_e(r,u_{r_j})=\phi_e(rr_j,u).$$

By Lemma \ref{pm}, we know that $(\partial_e u)^{\pm}$ and $(\partial_e u_\infty)^{\pm}$ satisfy the assumptions in ACF monotonicity formula (Theorem \ref{ACF}), for any unit vector $e$ such that $e\perp e_1$. Thus we know that the limit $\phi_e(\infty,u)$ exists and 
$$\phi_e(r,u_\infty)=\lim_{j\to \infty}\phi_e(rr_j,u)=\phi_e(\infty,u), $$
for all $r>0$ and $e \perp e_1$, i.e., $\phi_e(r,u_\infty)$ is constant for all $r>0$ and $e\perp e_1$. By Theorem \ref{cas ACF},  either one of the following holds for $e\perp e_1$:
\begin{itemize}
 \item[(i)] $(\partial_eu_\infty)^+\equiv 0$ or $(\partial_eu_\infty)^-\equiv 0$ in $\re^n$;
 \item[(ii)] there exists a unit vector $w=w(e)$ and constants $k_{\pm}=k_{\pm}(e)>0$ such that
$$(\partial_eu_\infty)^+=k_+(x\cdot w)^+, \quad (\partial_eu_\infty)^-=k_-(x\cdot w)^- \quad \forall  x \in \re^n.$$
\end{itemize}
Since $|Int \Lambda(u_\infty)| \neq 0,$
we know that
 $(ii)$ does not hold for any direction $e \perp e_1$, i.e., we know that $(i)$ holds for any direction $e \perp e_1$.
Consequently, we have that
$$0 \le \phi_e(r,u) \le \phi_e(\infty,u)=\phi_e(r, u_\infty)=0,$$
for any $r>0$ and $e\perp e_1$. 
Then again, by $|Int \Lambda (u)| \neq 0$ and Theorem \ref{cas ACF}, we know that $\partial_e u$ has a sign for all $e\perp e_1$, i.e., 
$$\partial_eu\ge 0\quad \text{ or } \quad \partial_eu \le 0 \quad \text{ in } \re^n \text{ for any } e\perp e_1.$$
By Lemma \ref{lem 2dim}, in an appropriate system of coordinates
$$u(x)=w(x_1,x_2), \quad x\in \re^n,$$
with $\partial_2 w\ge 0$.
\end{proof}

\begin{lem}\label{lem 2dim}
If $u\in C^1(\re^n)$ and if $\partial_e u$ does not change sign in $\re^n$, where $e\perp e_1$, then there exist a function $w\in C^1(\re^2)$ and a direction $\tilde e\perp e_1$ such that
$$u(x)=w(x_1,x\cdot \tilde e), \quad x\in \re^n$$
where $w$ is a monotone function with the second variable.
\end{lem}

\begin{proof}
The obvious proof is left to the reader.
\end{proof}

\begin{prop}\label{pos glo}
Let  $u \in P_\infty(M)$ with the upper obstacle  
$$\psi(x)=\frac{a}{2}(x_1^+)^2 \quad \text{ in } \re^n.$$
Assume that there exists $\epsilon_0>0$ such that
$$ \delta_r(u) \ge \epsilon_0 \quad \forall r>0.$$
Then $0\le u$ in $\re^2$ and $u$ is a solution of 
\begin{equation}\label{2d po}
\Delta u = \chi_{\{0<u<\psi\}} + a \chi_{\{0<u=\psi\}}, \quad 0\le u \le \psi \quad \text{  a.e. in } \re^2.
\end{equation}
\end{prop}


\begin{proof}
By Lemma \ref{2dim}, we know $u$ is a 2-dimensional function and $|Int\Lambda(u)|\neq 0$ and in an appropriate system of coordinates
$$\partial_2 u(x)\ge 0 \quad \forall x\in \re^2.$$
Thus we know that there is a ball $B_\delta(x_0) \subset \Lambda(u)$ and $u\le 0$ in
$$K(x^0,\delta)=\{(x_1, x_2-m) | (x_1,x_2)\in B_\delta(x^0), m \ge0\}.$$
Since $u$ is subharmonic, by the strong maximum principle, we obtain
$$u\equiv 0 \quad \text{ in } K(x^0, \delta).$$

By the assumption, $\partial_2 u \ge 0$ in $\re^2$, we know that the limit, $\lim_{x_2 \to -\infty} u(x_1,x_2)$ exists, for all $x_1\in \re^1$. Then we define a 1-dimensional function 
$$\hat u(x_1):=\lim_{x_2 \to -\infty} u(x_1,x_2).$$
Since $K(x^0,\delta) \subset \Lambda(u)$, we obtain
$$|u(x_1,x_2)|\le \frac{M}{2}|x_1-x^0_1|^2,$$
where $x_2 \le x^0_2$,
and therefore
$$|\hat u(x_1)|\le \frac{M}{2}|x_1-x^0_1|^2,$$
and $\hat u(x_1)$ is finite for any $x_1\in \re^1$.

 By the definition of $\hat u$ and the fact that $u(x_1,x_2-t)$ is a solution of \eqref{2d po} for all $t>0$, we know that $\hat u$ is a limit of the solutions of \eqref{2d po} and $\hat u$ is a solution of the obstacle problem with upper obstacle $\psi(x_1)=\frac{a}{2}(x_1^+)^2$ in $\re^1$. 

By the definition of $\hat u$ and $K(x^0, \delta) \subset \Lambda(u)$, we know $B'_\delta(x^0_1) \subset \Lambda(\hat u).$ Suppose that the connected component of $\Lambda(\hat u)$ containing $B'_\delta(x^0_1)$ is a closed interval, $[\alpha, \beta]\subset \re^1$ (call it $\tilde \Lambda(\hat u)$). By the non-degeneracy, we know that there are points $\alpha_0$ and $\beta_0$ such that $\alpha_0< \alpha< \beta< \beta_0$ and $\hat u(x)>0$ for all $x\in (\alpha_0,\alpha) \cup (\beta, \beta_0).$ Thus, if there is a point $z$ such that $\hat u(z)<0$, then there is an open interval $I$ such that $\hat u>0$ on $I$ and $\hat u=0$ at the ends points of $I$. By the maximum principle, however, $\hat u \le 0$ on $I$. Thus, we arrive at a contradiction. 
In the case that $\tilde \Lambda( \hat u)$ is $(-\infty,\alpha]$ or $[\beta, \infty)$ for some $\alpha, \beta \in \re^1$, we also have the same contradiction. 
Therefore we obtain  $\hat u \ge 0$ in $\re^1$.

By the definition of $\hat u$ and $\partial_2 u \ge 0$ in $\re^n$, we obtain
$$u(x_1,x_2) \ge \hat u(x_1) \ge 0 \quad \forall x=(x_1,x_2)\in \re^2,$$
and $u$ is a solution of 
\begin{equation*}
\Delta u = \chi_{\{0<u<\psi\}} + a \chi_{\{0<u=\psi\}}, \quad 0\le u \le \psi \quad \text{  a.e. in } \re^2.
\end{equation*}
\end{proof}

\subsection{Homogeneity of Blowup and Shrink-down of Global Solutions with the Upper Obstacle $\psi=\frac{a}{2}(x_1^+)^2$}
\textup{In order to deal with homogeneity, we introduce Weiss' energy functional for the problem \eqref{main eq}. It is a modification of Weiss' energy functional for the classical obstacle problem, 
$\Delta u =\chi_{\{u>0\}}$, $u \ge 0$ in $B_R$, and has already appeared in \cite{Ale}. We give the proof for reader's convenience. 
}
 
\begin{definition}\label{def wei}
Let  $u\in P_R(M)$ be a solution of
$$\La u=\chi_{\{0<u<\psi \}}+\La \psi \chi_{\{0<u=\psi\}} \quad  \text{ on } B_R,$$
with the upper obstacle $$\psi(x)=\frac{a}{2}( x_1^+)^2.$$
We define Weiss' energy functional for $u$ and $0<r<R$ as
\begin{align*}
W(r,u):&=\frac{1}{r^{n+2}}\int_{B_r}(|Du|^2+2u\La u)dx-\frac{2}{r^{n+3}}\int_{\partial B_r} u^2 dH^{n-1}\\
&=\frac{1}{r^{n+2}}\int_{B_r}|Du|^2dx + \int_{B_r\cap \{\psi>u>0\} } 2udx+ \int_{B_r\cap \{\psi=u>0\} } 2audx\\
&-\frac{2}{r^{n+3}}\int_{\partial B_r} u^2 dH^{n-1}
\end{align*}
\end{definition}

\begin{theorem}[Weiss' monotonicity formula]
Let $u, \psi$ be as in Definition \ref{def wei}. Then $r \to W(r,u)$ is a nondecreasing absolutely continuous function for $0<r<R$ and
\begin{equation*}
\frac{d}{dr}W(r,u) = \frac{2}{r^{n+4}}\int_{\partial B_r}|x \cdot Du(x)-2u(x)|^2 dH^{n-1},
\end{equation*}
for a.e. $0<r<R$. Furthermore, if $W(r,u)$ is constant for $r>0$, then $u$ is homogeneous of degree two, i.e.,
$$u(\lambda x)= \lambda^2u(x) \quad \text{ for all } x\in \re^n, \lambda >0.$$
\end{theorem}

\begin{proof}
By the scaling property $W(r,u)=W(1,u_r)$, we have
\begin{align*}
\frac{d}{dr}W(r,u)&=\frac{d}{dr}W(1,u_r)\\
&=\int_{B_1}\frac{d}{dr}(|Du_r|^2)dx + \int_{B_1\cap \{\psi_r>u_r>0\} } 2\frac{d}{dr}(u_r)dx\\
&+\int_{B_1\cap \{\psi_r=u_r>0\} } 2a\frac{d}{dr}(u_r)dx-2\int_{\partial B_1} \frac{d}{dr}(u_r^2) dH^{n-1}.\\
\end{align*}

Since $\dfrac{d}{dr}(\nabla u_r)=\nabla \dfrac{du_r}{dr}$ and $\dfrac{du_r}{dr}=\dfrac{x \cdot \nabla u_r-2u_r}{r},$ we obtain, by integration by parts,

\begin{align*}
\frac{d}{dr}W(r,u)&=2\int_{B_1}- \La u_r \dfrac{du_r}{dr}dx + \int_{B_1\cap \{\psi_r>u_r>0\} } 2\frac{du_r}{dr}dx \nonumber \\ 
&+\int_{B_1\cap \{\psi_r=u_r>0\} } 2a\frac{du_r}{dr}dx+2\int_{\partial B_1}(\partial_\nu u_r -2u_r)\dfrac{du_r}{dr} dH^{n-1}\\ \nonumber
&=2r\int_{\partial B_1}\left|\dfrac{du_r}{dr}\right|^2 dH^{n-1}.
\end{align*}
Then we have the desired equality after scaling.
\end{proof}

\begin{cor} (Homogeneity of blowup and shrink-down) \label{hom}
Let  $u\in P_\infty(M)$ be a solution of
$$\La u=\chi_{\{0<u<\psi \}}+\La \psi \chi_{\{0<u=\psi\}} \quad  \text{ on } \re^n,$$
with the upper obstacle $$\psi(x)=\frac{a}{2}( x_1^+)^2.$$
Then any blowup function $u_0$ of $u$ at $0$ and any shrink-down $u_\infty$ of $u$ at $0$ are homogeneous of degree two. 
\end{cor}

\begin{proof}
Suppose that $\lambda_j \to 0$ as $j \to \infty$ and $u_{\lambda_j} \to u_0$ in $C^{1,\alpha}_{loc}(\re^n)$ as $j \to \infty$. Then for $r>0$
$$W(r,u_0)=\lim_{j \to \infty} W(r,u_{\lambda_j})=\lim_{j \to \infty} W(\lambda_j r,u)=W(0+,u),$$
i.e., $W(r,u_0)$ is constant for any $r$.
Hence, $u_0$ is homogeneous of degree two.\\
\indent In order to prove the homogeneity for shrink-down $u_\infty$, we take a sequence $\lambda'_j \to \infty$ as $j \to \infty$ and $u_{\lambda'_j} \to u_\infty$ in $C^{1,\alpha}_{loc}(\re^n)$ as $j \to \infty$. The same argument as above shows that $W(r,u_\infty)$ is constant for any $r>0$ and the homogeneity of shrink-down.
\end{proof}

\textup{Under the conditions of Theorem \ref{reg lo1}, we know that the blowups and shrink-downs of the blowups $u_0$ of $u\in P_1(M)$ are two-dimensional and homogeneous of degree two, see the proof of Proposition \ref{cla glo}. For further study on the main theorem, we need to know about the global solutions which are two-dimensional and homogeneous of degree two.}

\begin{lem}\label{cla hom} 
Let $u\in P_\infty(M)$ and $u$ is a solution of 
$$\Delta u = \chi_{\{0<u<\psi\}} + a \chi_{\{0<u=\psi\}}, \quad 0\le u \le \psi \quad \text{  a.e. in } \re^2,$$ 
with the upper obstacle
$$\psi(x)=\frac{a}{2}( x_1^+)^2,$$
for a constant $a>1$.
Suppose that $u$ is homogeneous of degree two. Then
$$u(x)=\frac{1}{2}( x_1^+)^2  \quad \text{ or } \quad u(x)= \frac{a}{2}( x_1^+)^2.$$
\end{lem}

\begin{proof}
By the condition $0\le u \le \psi$ and $\psi(x)=\frac{a}{2}( x_1^+)^2$, we know that $\{x_1<0\} \subset \{u=0\} $. We claim that $\partial_2 u \equiv 0$ in $\re^2$, i.e., $u$ is one-dimensional function.

Assume that $\{\partial_2 u \neq 0\}\cap \{x_1>0\}\neq \emptyset$. Then by the homogeneity of degree one for $\partial_2 u$,
 we know that there is a cone 
$$\mathcal C:=\{ r \theta \text{ }|\text{ } r>0, \alpha_1<\theta < \alpha_2 \}\subset \{x_1>0\} ,$$
 ($-\frac{\pi}{2}\le \alpha_1 < \alpha_2 \le \frac{\pi}{2}$) such that
$\partial_2u\neq 0$ in $C$ and $\partial_2 u =0$ on $\partial C.$ Since $ \partial_2 \psi\equiv 0$, we know that 
$$\mathcal C \subset \{0<u < \psi \}\cap \{x_1>0\}$$
and $\partial_2 u$ is harmonic on $\mathcal C$.
Hence $\partial_2 u:=rf(\theta)$ satisfies
$$\Delta \partial_2 u=\Delta \left(rf(\theta) \right)=\frac{1}{r}\left( f(\theta)+f''(\theta) \right)=0 \quad \text{ on } \mathcal C.$$
Thus $f(\theta)$ satisfies $-f''(\theta)=f(\theta)$ in $(\alpha_1, \alpha_2)$ and $f(\theta)=0$ on $\partial (\alpha_1, \alpha_2)$. Hence we obtain
$f(\theta)=c \cos(\theta)$ in $(-\frac{\pi}{2}, \frac{\pi}{2})$, $C=\{ r \theta \text{ }|\text{ } r>0, -\frac{\pi}{2}<\theta < \frac{\pi}{2} \}=\{x_1>0\}$  and 
$$\partial_2 u= cr\cos(\theta)= c x_1 \text{ in }\{x_1>0\}.$$
 Then $u=cx_1x_2$ in $\{x_1>0\}=\{0<u<\psi\}.$ It is a contradiction to $\Delta u=1$ in $\{0<u< \psi\}.$ Hence we obtain that $\partial_2 u \equiv 0$ in $\re^2$. This completes the proof.
\end{proof}

\section{Directional Monotonicity}\label{sec dir mon}

\textup{In this section, we prove the directional monotonicity for solutions to \eqref{main eq}. The proofs in this section follow standard patterns as that of classical obstacle problem but one still needs some care. Hence, we shall give some details. Let us start with the following lemma, where the proof is exactly the same as that of  Lemma 4.1 of \cite{PSU}, and hence omitted.
}

\begin{lem}\label{upos}
 Let $u\in P_1(M)$ and $f\ge c >0$ in $B_1$, $\Delta \psi \ge c>0$ in $\Omega(\psi)$ and any blowup $u_0$ satisfies
$$u_0(x)=\frac{1}{2} (x_1^+)^2 \quad \text{ or } \quad u_0=\frac{a}{2}(x^+_1)^2.$$
 Suppose, further, that
$$ \norm{u-u_0}_{L^\infty(B_1)}\le \epsilon.$$
Then
$$u>0 \quad \text{ in } \{x_1>\sqrt{2\epsilon}\} \cap B_1,$$
$$u=0 \quad \text{ in } \left\{x_1\le -4\sqrt{ \frac{n\epsilon}{c}}\right\} \cap B_{1/2}.$$
\end{lem}

\begin{lem}\label{inc dri}
 Let $u\in P_1(M)$ and $f\ge c >0$ in $B_1$, $\La \psi \ge c>0$ in $\Omega(\psi)$. Suppose that we have
$$C\partial_e \psi -\psi \ge -\epsilon_0, \quad C\partial_e u -u\ge -\epsilon_0 \quad \text{ in } B_1,$$
for a direction $e$ and $\epsilon_0<c/64n$. Then we obtain
$$C\partial_e \psi -\psi \ge 0 \quad \text{ in } B_{3/4}, \quad C\partial_e u -u \ge 0 \quad \text{ in } B_{1/2},$$
where $\norm{Df}_{L^{\infty}(\overline{B_1})}, \norm{D^3\psi}_{L^{\infty}(\overline{\Omega(\psi)\cap B_1})}<\dfrac{c}{2C}$.
\end{lem}

\begin{proof}
First, we will prove $$C\partial_e \psi -\psi \ge 0 \quad \text{ in } B_{3/4}.$$
Arguing by contradiction, suppose there is a point $y\in B_{3/4} \cap \Omega(\psi)$ such that $C\partial_e \psi(y) -\psi(y) < 0.$ 
Define the auxiliary function
$$\phi(x)=C\partial_e \psi(x) -\psi(x)+\frac{c}{4n}|x-y|^2.$$
Then
\begin{align*}
\La \phi(x)&=C\La \partial_e \psi(x) -\La\psi(x)+\frac{c}{2}\\
&\le C\norm{D^3 \psi}_{L^{\infty}(\overline{\Omega(\psi)\cap B_1})} -\La\psi(x)+\frac{c}{2}\\
&\le c-\La \psi \le 0  
\end{align*}
$\text{on } B_{1/4}(y) \cap \Omega (\psi).$ Since $\phi(y)<0$, by the minimum principle, $\phi$ has the negative infimum on $\partial (B_{1/4}(y) \cap \Omega(\psi))$. Since $\phi \ge0$ on $\partial \Omega(\psi)$, we have 
$$\inf_{\partial B_{1/4}(y) \cap\Omega(\psi)} \phi <0.$$
It is equivalent to
$$\inf_{\partial B_{1/4}(y) \cap\Omega(\psi)}\left( C\partial_e \psi-\psi \right) <-\frac{c}{64n}.$$
Since $\epsilon_0<c/64n$, we have a contradiction. 

By using $C\partial_e \psi -\psi \ge 0 \text{ in } B_{3/4}$, $\{u=\psi\}=\{u=\psi\}\cap \{\nabla u=\nabla \psi\}$ (since $u \le \psi$) and the same method as above, we have
$$C\partial_e u -u \ge 0 \quad \text{ in } B_{1/2}\cap \Omega(u)\cap \{ u<\psi\}.$$
This completes the proof of the lemma.
\end{proof}

\begin{lem}\label{inc dri'}
Let $u \in P_1(M)$ and $f\ge c >0$ in $B_1$, $\La \psi \ge c>0$ in $\Omega(\psi)$. Let further
$$ \psi_0=\frac{a}{2}(x_1^+)^2$$ 
and
$$u_0(x)=\frac{1}{2} (x_1^+)^2 \quad \text{ or } \quad u_0=\frac{a}{2}(x^+_1)^2.$$
Suppose also $\norm{Df}_{L^{\infty}(\overline{\{\psi>0\}\cap B_1})}, \norm{D^3\psi}_{L^{\infty}(\overline{\{\psi>0\}\cap B_1})}<\dfrac{c\delta}{2}$ for $0<\delta \le 1$ and
\begin{equation}\label{clo}
\norm{u-u_0}_{C^{1}(B_1)}, \quad \norm{\psi-\psi_0}_{C^{1}(B_1)} \le \epsilon.
\end{equation}
Then $\epsilon\le \dfrac{c}{128n}$ implies 
\begin{align*}
u &\ge 0\quad \text{ in } B_{1/2},
\end{align*}
and $\epsilon\le \dfrac{c \delta}{128n}$ implies
\begin{align*}
\partial_e u &\ge0 \quad \text{ in } B_{1/2}, 
\end{align*}
 for any
$$e \in \mathcal C_\delta \cap \partial B_1,$$
where
$$\mathcal C_\delta=\{x\in \re^n : x_1 > \delta |x'|\}, \quad x'=(x_2,...,x_n).$$
\end{lem}

\begin{proof}
Direct computation shows that
$$\delta^{-1} \partial_e u_0-u_0 \ge 0, \quad \delta^{-1} \partial_e \psi_0-\psi_0 \ge 0 \text{ in } B_1 \quad \text{ for any } e\in \mathcal C_\delta\cap \partial B_1.$$ 
By using the closeness condition \eqref{clo} for $\epsilon\le c \delta/128n$,  we have
$$\delta^{-1}\partial_e u -u\ge-2\epsilon \delta^{-1}\ge -\frac{c}{64n}, \quad  \delta^{-1}\partial_e \psi -\psi \ge -\frac{c}{64n} \quad \text{ in } B_1.$$

By Lemma \ref{inc dri}, we have
\begin{equation}\label{mon di}
\delta^{-1}\partial_e u -u\ge 0 \quad \text{ in } B_{1/2} \quad \text{ for any } e\in \mathcal C_\delta\cap \partial B_1.
\end{equation}
Recalling Lemma \ref{upos}, we have
$$u=0 \quad \text{ in } \left\{x_1\le -\frac{1}{2\sqrt{2}} \right\} \cap B_{1/2}.$$
Let $\delta=1$ and multiply \eqref{mon di} by $exp(-e\cdot x)$. Then we have
$$\partial_e (exp(-e\cdot x)\cdot u)\ge 0 \quad \text{ in } B_{1/2}.$$
By integrating $(exp(-e\cdot x)\cdot u)$ with direction $e\in C_1$, we obtain $u \ge 0$ in $B_{1/2}$. Moreover, we have that 
$\partial_e u \ge0  \text{ in } B_{1/2},$ for any $e \in \mathcal C_\delta \cap \partial B_1.$
\end{proof}
\emph{
The rescaled function $u_r$ at $0$ satisfies 
$$ \La u_r=f(rx) \chi_{\{\psi_r>u_r>0\}}+\La \psi(rx) \chi_{\{\psi_r=u_r>0\}} \quad \text{ in } B_{1/r}.$$
Moreover, when  $r$ tends to $0$, then $u_r$ converges to $u_0$ in $C^{1,\alpha}_{loc}(\re^n)$ and
$$\norm{D (f(rx))}_{L^{\infty}(B_1)}=r\norm{D f(rx)}_{L^{\infty}(B_1)} \le r\norm{D f(x)}_{L^{\infty}(B_1)},$$
$$\norm{D (\La\psi(rx))}_{L^{\infty}(B_1)}=r\norm{D\La \psi(rx)}_{L^{\infty}(B_1)} \le r\norm{D \La \psi(x)}_{L^{\infty}(B_1)}$$
converge to $0$. Therefore, we have the following lemma.
}
\begin{lem}\label{dir mon}(Directional monotonicity)
Let $u \in P_1(M)$ and $f\ge c >0$ in $B_1$, $\La \psi \ge c>0$ in $\Omega(\psi)$. 
Let 
$$ \psi_0=\frac{a}{2}(x_1^+)^2$$ and
$$u_0(x)=\frac{1}{2} (x_1^+)^2 \quad \text{ or } \quad u_0=\frac{a}{2}(x^+_1)^2,$$
 where $u_0$ and $\psi_0$ are blowup functions of $u$ and $\psi$, respectively. Then for any $\delta \in (0,1]$ there exists $r_\delta=r(\delta,u)>0$ such that
\begin{align*}
 u &\ge 0 \quad \text{ in } B_{r_1}\\
\partial_e u &\ge 0 \quad \text{ in } B_{r_\delta}\quad \text{ for any } e\in \mathcal C_\delta.
\end{align*}
\end{lem}

\section{Classification of Blowups}

\emph{In this section, we classify the blowups by using the results in Section \ref{sec pro glo}, \ref{sec dir mon}.}

\begin{prop}\label{cla glo}
Let $u \in P_1(M)$ with an upper obstacle $\psi$ such that 
$$0\in \partial \Omega(\psi), \quad \lim_{x \to 0, x\in \Omega(\psi)} \Delta \psi(x)=a>f(0)=1, \quad f\ge c >0 \text{ in } B_1,$$
and 
$$ \inf{\left\{\Delta \psi,  \Delta \psi-f\right\}} \ge c>0 \text { in }\Omega(\psi).$$
Suppose
$$\min{\left\{\delta_r(u), \delta_r(\psi)\right\}}\ge \epsilon_0 \quad \forall r< 1/4.$$
Then
$$\psi_0=\frac{a}{2}(x_1^+)^2 \quad \text{ and } \quad u_0=\frac{1}{2}(x_1^+)^2 \quad \text{ in } \re^n,$$
in an appropriate system of coordinates.
\end{prop}

\begin{proof}
Let $u_0$, $\psi_0$ be a global solution of $u$, $\psi$, respectively. Then $\psi_0$ is a global solution of 
$$\Delta\psi_0 =a \chi_{\Omega(\psi_0)} \quad \text{ in } \re^n,$$
with the thickness assumption,
$$\delta_r(\psi_0)> \epsilon_0, \quad \forall r>0.$$
By the non-degeneracy for $\psi$ (the proof is almost the same as that of Lemma \ref{nond}), we know $0\in \Gamma(\psi_0)$; see also Proposition 3.17 (iv) in \cite{PSU}. By Theorem \RN{2} of \cite{CKS}, we obtain that $\psi_0$ is a half-space solution, i.e.,
$$\psi_0=\frac{a}{2}(x_1^+)^2 \quad \text{ in } \re^n,$$
in an appropriate system of coordinates.
By Proposition \ref{pos glo}, $u_0$ is two-dimensional,  $u_0(x) =u_0(x_1,x_2)$, and hence  a solution of 
$$\Delta u_0=\chi_{\{0<u_0< \psi_0\}}+a\chi_{\{0<u_0=\psi_0 \}}, \quad 0\le u_0\le \psi_0 \quad \text{ a.e. in } \re^2.$$ 
Let $u_{00}=(u_0)_0$ and $u_{0\infty}=(u_0)_\infty$ be blowup and respectively shrink-down of $u_0$ at $0$. By Corollary \ref{hom}, $u_{00}, u_{0\infty}$ are homogeneous of degree two and by Lemma \ref{cla hom},
$$u_{00}=\frac{1}{2}(x_1^+)^2 \quad \text{ or }\quad \frac{a}{2}(x_1^+)^2\quad\text{and} \quad  u_{0\infty}=\frac{1}{2}(x_1^+)^2 \quad \text{ or } \quad  \frac{a}{2}(x_1^+)^2.$$
By Lemma \ref{upos}, \ref{inc dri} and \ref{inc dri'} for $u_{00}$ and $u_0$ and the fact that $(u_0)_r$ converges to $u_{00}$ as $r\to 0$ in $C^{1,\alpha}_{loc}(\re^n)$, we know there are $r',\epsilon'>0$ such that 
\begin{align} 
\delta^{-1}\partial_eu_0 -u_0 \ge 0 \quad &\text{ in } B_{r'} \quad \text{ for any } e\in \mathcal C_\delta \cap \partial B_1. \label{con u_01} \\
u_0=0 \quad & \text{ in } \{x_1 < -\epsilon'\} \cap B_{r'}. \label{con u_02}
\end{align}
Moreover, by Lemma \ref{upos}, \ref{inc dri} and \ref{inc dri'} for $u_0$ and $u$ in $B_{r'}$ with the conditions, \eqref{con u_01} and \eqref{con u_02}, we know that there is $r''$ such that
$$u \ge 0 \quad \text{ in } B_{r''}.$$
Then we know that $0\in \Gamma^{\psi_0}(u_0)$ and $0\in \Gamma^{\psi_0}(u_{00}), \Gamma^{\psi_0}(u_{0\infty})$ (see Remark \ref{nonde for v}).
Thus we obtain 
$$u_{00}=u_{0\infty}=\frac{1}{2}(x_1^+)^2.$$
Since 
\begin{align*}
W(1,u_{00})=\lim_{r \to 0} W(1,(u_{0})_r)=&\lim_{r \to 0} W(r,u_{0})\\
\le &\lim_{r \to \infty } W(r,u_0)=\lim_{r \to \infty} W(1,(u_{0})_r)=W(1,u_{0\infty})
\end{align*}
and 
$W(1, u_{00})=W(1, u_{0\infty}),$
we know that $W(r, u_0)$ is constant for $r>0$.
Hence, by Lemma \ref{cla hom} and $0\in \Gamma^\psi(u_0)$, we know that $u_0$ is homogeneous of degree two and
$$u_0(x)=\frac{1}{2}(x_1^+)^2.$$
\end{proof}

\section{Proof of Theorem \ref{reg lo1}}\label{sec reg loc}

\emph{
Let $u$ be as in Proposition \ref{cla glo}. Then a blowup function $u_0$ of $u$ at $0$ is a half-space solution, i.e.,
$$u_0=\frac{1}{2}(x_1^+)^2,$$
in an appropriate system of coordinates. By the directional monotonicity for $u$ (Lemma \ref{dir mon}), we have the uniqueness of blowup (see Proposition 4.6 of \cite{PSU}).
}
\begin{prop}[Uniqueness of blowup]
Let $u$ be as in Proposition \ref{cla glo}. Then the blowups of $u$ at $0$ is unique, i.e., in an appropriate system of coordinates,
for any sequence $\lambda_i \to 0$, 
$$u_{\lambda_i} \to u_0=\frac{1}{2}(x_1^+)^2 \quad  \text{ in } C^{1,\alpha}_{loc}(\re^n)$$ 
as $\lambda_i \to 0$.
\end{prop}

\begin{lem}\label{near zero}
Let $u$ be as in Proposition \ref{cla glo}. Then there is $r'_1=r'_1(u,\psi)>0$ such that the blowup function of $u$ at $x\in \Gamma(u) \cap B_{r'_1}$ are half-space functions.
\end{lem}

\begin{proof}
By Proposition \ref{cla glo}, we have the directional monotonicity for $u$ (see Lemma \ref{dir mon}).
 Moreover, by Lemma \ref{inc dri} we also have the directional monotonicity for $\psi$. 
Thus, for any $\delta \in (0,1]$, there exists $r_1\ge r'_\delta=r'_\delta(u,\psi)>0$ such that
\begin{align*}
 \psi,u &\ge 0 \quad \text{ in } B_{r'_1}\\
\partial_e \psi, \partial_e u &\ge 0 \quad \text{ in } B_{r'_\delta}\quad \text{ for any } e\in \mathcal C_\delta.
\end{align*}
Hence, by the sign condition $u \ge 0 $ in $B_{r'_1}$, we know that  $u$ is a solution of 
$$\Delta u =f\chi_{\{ 0< u < \psi \}} +\Delta \psi \chi_{\{0< u = \psi\}}, \quad 0\le u \le \psi \quad \text{ in } B_{r'_1}$$
and the free boundaries $\partial \{u=0\}\cap B_{r'_1}=\Gamma(u) \cap B_{r'_1}$ and $\partial \{\psi=0\}\cap B_{r'_1}$ are represented by Lipschitz functions; for details, see Proposition 4.8 of \cite{PSU}.

Case 1) Let $x^0\in\Gamma(u) \cap  B_{r'_1}=\partial \{u=0\} \cap B_{r'_1}$ and assume that there exists $r_0>0$ such that 
$$\{u=\psi\}\cap B_r(x^0)\neq \emptyset \quad \forall r< r_0.$$
Then we can find a sequence of points $x^j\in \{u=\psi\}$ converging to $x^0$ as $j \to \infty.$ Then we have
$$\psi(x^j)=u(x^j)\to 0 \quad \text{ as } j \to \infty,$$
i.e., $x^0\in \{\psi=0\}$. By the sign condition $0\le u \le \psi$ in $B_{r'}$, we know $\{\psi=0\}\subset \{u=0\}$ in $B_{r'_1}$ and therefore $x^0\in\partial \{u=0\} \cap  B_{r'_1} $ implies $x^0\in \partial \{\psi=0\}.$ On the other hand, Lipschitz regularity of $\partial \{u=0\}$ and $\partial \{\psi=0\}$ implies the thickness condition for $\psi$ and $u$, i.e., for some $\epsilon_0, \tilde r=\tilde r(x^0)>0$,
$$\min{\left\{\delta_r(u), \delta_r(\psi) \right\}} \ge \epsilon_0>0 \quad \forall \tilde r\ge r>0.$$ 
Then, by Proposition \ref{cla glo}, we know that the blowup function of $u$ at $x^0$ is a half-space solution (we may assume $\lim_{x \to x^0, x\in \Omega(\psi)} \Delta \psi(x)>f(x^0)$, by the conditions  $\psi \in C^{1,1}(B_1) \cap  C^{2,1}(\overline{\Omega(\psi)})$, $f\in C^{0,1}(B_1)$ and $\lim_{x \to 0, x\in \Omega(\psi)} \Delta \psi(x)=a>f(0)=1$). 

Case 2) Let $x^0\in \Gamma(u)$ and assume that there exists $r_0>0$ such that 
$$\{u=\psi\}\cap B_{r_0}(x^0)= \emptyset.$$
 Then $u$ is a solution of an obstacle problem
$$\Delta u = f\chi_{\{u>0\}}, \quad u\ge0 \quad \text{ in } B_{r_0}(x^0).$$
By Theorem \RN{2} of \cite{CKS} and the thickness condition for $u$ at $x^0$, we know that that the blowup function of $u$ at $0$ is a half-space solution.
\end{proof}

\begin{proof}[Proof of Theorem \ref{reg lo1}]
By Proposition \ref{cla glo}, we have the directional monotonicity for $u$ (see Lemma \ref{dir mon}). Thus, we know that the free boundary $\Gamma(u)\cap B_{r\delta/2}$ is represented as a graph $x_n=f(x')$ with Lipschitz constant of $f$ not exceeding $\delta$. Since $\delta>0$ is arbitrary, we have a tangent plane of $\Gamma(u)$ and the normal vector $e_n$  at $0$.  By Lemma \ref{near zero}, we know that every point $z\in \Gamma(u) \cap B_{r'_1}$ has a tangent plane. Moreover again, by using the directional monotonicity, we obtain that $\Gamma(u) \cap B_{r'_1}$ is $C^1$ (see Theorem 4.10 of \cite{PSU}).

We know that there is a ball $B_{r'_1}$ such that $u \ge0$ in $B_{r'_1}$ and $v=\psi-u$ is a solution of 
$$\Delta v =(\Delta \psi-f)\chi_{\{ 0< v < \psi \}} +\Delta \psi \chi_{\{0< v = \psi\}}, \quad 0\le v \le \psi \quad \text{ in } B_{r'_1}$$
and the blowup function $v_0$ of $v$ at $0$ is a halfspace solution. Thus we have the directional monotonicity for $v$ and $C^1$ regularity of the free boundary $\Gamma(v)=\Gamma^\psi(u)$ near $0$ by using the same method as that in the above paragraph.
\end{proof}

\end{document}